\documentclass[11pt,a4paper]{article}
\usepackage[]{geometry}
\usepackage{sectsty}
\sectionfont{\normalsize\bfseries}
\subsectionfont{\normalsize\sf}
\usepackage[mathcal]{euscript}
\usepackage{bm,url}                     
\usepackage{amsfonts,todonotes}
\usepackage{amssymb}
\usepackage{amsmath}
\usepackage{amsthm}
\usepackage{graphicx} 
\usepackage{subcaption}              
\usepackage{natbib}                 
\setcitestyle{authoryear,open={(},close={)}}                 
\usepackage{colortbl}               
\usepackage{booktabs}
\newcommand{\R}[1]{\mathbb{R}^{#1}}

\usepackage[english]{babel}
\usepackage{amsmath}
\usepackage{amsfonts}
\usepackage{amssymb}
\usepackage{amsthm}
\usepackage{graphicx}
\usepackage{color}
\usepackage{array}
\usepackage{mathrsfs}
\usepackage{setspace}
\usepackage{hyperref}
\usepackage[super]{nth}

\definecolor{violet}{rgb}{0.7,0,0.6}

\usepackage{amsfonts}
\usepackage{amssymb}
\usepackage{amsmath}
\usepackage{amsthm}
\newcommand{\E}{\mathcal{E}}

\newtheorem{proposition}{{\sc\bf Proposition}}
\newtheorem{theorem}{{\sc\bf Theorem}}

\newtheorem{remark}{{\sc\bf Remark}}

\newtheorem{assumption}{{\sc\bf Assumption}}

\def\P{{\mathbb P}}
\def\E{{\mathbb E}}
\def\R{{\mathbb R}}

\def\P{\mathbb{P}}

\def\cals_+{{\cals_+}}

\def\calb{{\mathcal{B}}}

\def\calh{{\mathcal{H}}}
\def\call{{\mathcal{L}}}
\def\cals{{\mathcal{S}}}

\newcommand{\var}{{\rm Var}}

\newcommand{\dd}{\mathrm{d}}
\newcommand{\wh}{\widehat}

\newcommand{\operK}{\mathcal{K}}

\allowdisplaybreaks 
\definecolor{violet}{rgb}{0.7,0.2,0.6}

\begin{document}

\begin{center}
	\Large \bf On functional logistic regression: some conceptual issues
\end{center}
 \normalsize
 \begin{center}
 	Jos\'e R. Berrendero$^1$\hspace{.1cm} Beatriz Bueno-Larraz$^2$ and \hspace{.1cm} Antonio Cuevas$^1$ \\
 	$^1$ Departamento de Matemáticas, Universidad Autónoma de Madrid\\
 	$^2$ Independent Data Scientist
 \end{center}

\

\begin{abstract} 
\footnotesize {The main ideas behind the classical multivariate logistic regression model make sense when  
 	translated to the functional setting, where the explanatory variable $X$ is a function and the response $Y$ is binary. However, some important technical issues appear (or are aggravated with respect to those of the multivariate case) due to the functional nature of the explanatory variable. First, the mere definition of the model can be questioned: while most approaches so far proposed rely on the $L_2$-based model, we suggest an alternative (in some sense, more general) approach, based on the theory of Reproducing Kernel Hilbert Spaces (RKHS). The validity conditions of such RKHS-based model, as well as its relation with the $L_2$-based one are investigated and made explicit in two formal results.  Some relevant particular cases are considered as well. Second we show that, under very general conditions, the maximum likelihood (ML) of the logistic model parameters fail to exist in the functional case. Third, on a more positive side, we suggest an RKHS-based restricted version of the ML estimator.  This is a methodological paper, aimed at a better understanding of the functional logistic model, rather than focussing on numerical and practical issues.}
\end{abstract}


\normalsize

\vspace*{5mm}

\section{Introduction: statement of the model}\label{sec:int}

\noindent \textit{Logistic regression: some basic ideas and references} 

\

Throughout this work we study the situation in which a binary (0-1) response variable
must be predicted  in terms of a random explanatory variable $X$, defined on a probability space $\Omega$. The logistic regression model is an extremely popular approach to such problem. The basic ideas of this model date back to the end of nineteenth century (a complete historical overview can be found in \citet[Ch. 9]{cramer2003}) but the logistic methodology is still under constant attention, especially as a (supervised) classification tool. The book by  \cite{hilbe2009} is a fairly complete reference about logistic regression.

The logistic model is a particular case of the wider family of generalized linear models (we refer to \cite{mccullagh1989} for details) which presents some interesting characteristics. According to \citet[p. 52]{hosmer2013}, one of its most appealing features is that the coefficients of the model are easily interpretable in terms of the values of the predictors. This technique stems from the attempt to apply well-known linear regression procedures to problems with categorical responses, like binary classification. There is no point in imposing that the categorical response is linear in the predictors $x$,   but we might instead assume  that $\log(p(x)/(1-p(x))$ is linear in $x$, where $p(x)={\mathbb P}(Y=1|X=x)$.  In this quotient the logarithm could be replaced with other link functions. However, an important aspect of  the logarithm-based  model is that it holds whenever the  predictor variable $X$ in both classes  is Gaussian with a common covariance matrix.

This finite-dimensional logistic model has been widely studied. Apart from the already mentioned references, \cite{efron1975} provides a comparison between logistic predictors and Fisher discriminant analysis. In addition, \cite{munsiwamy2011} gives a useful overview of asymptotic results of the estimators (firstly proved in \cite{fahrmeir1985} and \cite{fahrmeir1986}).

\

\noindent \textit{Logistic regression in the functional case: the ``classical'' $L^2$-model} 

\

The motivations for extending logistic regression to functional data are quite obvious given the current increasing availability of functional data in experimental sciences.  
An historical overview of several approaches to functional logistic regression can be found in \cite{mousavi2018}. 

We start by establishing the framework of the problem in this functional context. The goal is to explore the relationship between a dichotomous response variable $Y$,  taking values on $\{0,1\}$, and a functional predictor $X$. We will assume throughout that $X$ is an $L^2$-stochastic process with trajectories in $L^2[0,1]$. Thus, the random variable $Y$ conditional to the realizations $x$ of the process  follows a Bernoulli distribution with parameter $p(x)$ and the prior probability of class 1 is denoted by $p=\mathbb{P}(Y=1)$. In this setting, the most common functional logistic regression (FLR) model is
\begin{equation}\label{Eq:logitL2}
\P(Y=1 | X=x) = \frac{1}{1+\exp\{-\beta_0 - \langle \beta, x \rangle_2\}},
\end{equation}
where $\beta_0\in\R$, $\beta\in L^2[0,1]$ and $\langle \cdot, \cdot\rangle_2$ denotes the inner product in $L^2[0,1]$. This model is the direct extension of the $d$-dimensional one, where the product in $\R^d$ is replaced by its functional counterpart.

The standard approach to this problem is to reduce the dimension of the curves using Principal Components Analysis (PCA). That is, the curves are projected into the subspace defined by the eigenfunctions corresponding to the $d$  largest eigenvalues of the covariance operator. Then, we replace every curve in the functional data sample with the $d$-dimensional coordinates of the projections with respect to the basis formed by the eigenfunctions. 
Finally, standard  logistic regression is applied to the resulting $d$-dimensional vectors. Among others, this strategy has been explored by \cite{escabias2004} and \cite{james2002} from an applied perspective  though, in fact, the latter reference deals with generalized linear models   (and not only with logistic regression). These more general models are also studied by \cite{muller2005}, but with a more mathematical focus. 

\newpage

\noindent  \textit{Some preliminaries and notation for an alternative approach}

\

Our functional data will be trajectories in $L^2[0,1]$ of an $L^2$-process $X=X(t)$ with continuous covariance and mean function, denoted by  $K=K(s,t)$ and $m=m(t)$, respectively. The covariance operator $\operK$ associated with the covariance function $K$ of the process is given by
\begin{equation}\label{Eq:operK}
\operK (f)(\cdot) = \int_0^1 K(s,\cdot)f(s)\dd s = \E \big[ \langle X-m, f \rangle_2 \big( X(\cdot)-m(\cdot) \big) \big].
\end{equation}

 Since the approach we will explore here for functional logistic regression is based on the theory of Reproducing Kernel Hilbert Spaces (RKHS's), we will briefly remind here, for the sake of clarity, some basic ideas and notations about RKHS's; see \cite{berlinet2004} and Appendix F of \cite{janson1997} for further details and references).

Let $\calh_0(K):=\{f\in L^2[0,1] \ : \ f(\cdot)=\sum_{i=1}^n a_i K(t_i,\cdot),\ a_i\in{\mathbb R},\ t_i\in[0,1],\ n\in{\mathbb N}\}$, be 
the space of all finite linear combinations of evaluations of $K$.  This space is endowed with the inner product
$\langle f,g\rangle_K=\sum_{i,j}\alpha_i\beta_j K(t_i,s_j)$,
where $f(\cdot)=\sum_i\alpha_i K(t_i,\cdot)$ and $g(\cdot)=\sum_j\beta_j K(s_j,\cdot)$.

Then, the RKHS associated with $K$ is defined as the completion of $\calh_0(K)$. In other words, $\calh(K)$ is made of all functions obtained as pointwise limits of Cauchy sequences in $\calh_0(K)$. The inner product is extended accordingly to the whole space $\calh(K)$. 

These spaces are named after the so-called \textit{reproducing property}, 
$\langle f, K(s,\cdot)\rangle_K = f(s)$, for all  $f\in\calh(K), s\in[0,1]$, which is particularly important in the applications. On account of this property it is sometimes said that the RKHS are spaces of ``true functions'', in the sense that the pointwise values $f(s)$, at a given $s$ do matter, by contrast with $L^2[0,1]$ whose elements are in fact equivalence classes of functions. 

A property of RKHS's especially useful in statistical applications is given by the following isometry result:  let $L^2(\Omega) $ be the Hilbert space of real random variables with finite second moment, endowed with the usual inner product and with associated norm $\| U\|^2={\mathbb E}(U^2)$. Define
$$\call_0(X) = \big\{U\in L^2(\Omega) \ : \ U=\sum_{i=1}^n a_i \big(X(t_i)-m(t_i)\big),\ a_i\in{\mathbb R},\ t_i\in[0,1],\ n\in{\mathbb N}\big\},$$
where $m(t)=\E[X(t)]$, and let $\call(X)$ be the completion of $\call_0(X)$ in  $L^2(\Omega)$; hence, in other words, is a subspace of $L^2(\Omega)$ defined as the closure of the linear span of the centred one-dimensional marginals of the process $X$. It turns out that the transformation $\Psi_X$, from  $\call(X)$ to $\calh(K)$, defined by 
\begin{equation}\label{eq:isometry} 
\Psi_X(U)(s) = \E[ U(X(s)-m(s))] = \langle U, X(s)-m(s) \rangle \in \calh(K), \ \text{ for } U\in \call(X)
\end{equation} 
is an isometry (sometimes called \textit{Lo\`eve's isometry}) between $\call(X)$ and $\calh(K)$, that is, $\Psi_X(U)$ is bijective and preserves the inner product (see \citet[Lemma 1.1]{lukic2001}). As a consequence, the Hilbert spaces $\call(X)$ and $\calh(K)$ can be identified. 
Note that, in informal terms $\Psi_X$ is the completion of the transformation from $\call_0(X)$ to $\calh_0(K)$ given by $\sum_{i=1}^n a_i \big(X(t_i)-m(t_i)\big)\mapsto \sum_{i=1}^n a_i K(t_i,\cdot)$.

 It is worth mentioning that while $\calh(K)$ is, in several aspects, a natural Hilbert space associated with the process $X$, typically the trajectories of the process $X$ themselves do not belong to $\calh(K)$  with probability one (see, e.g., \cite[Cor. 7.1]{lukic2001}, \citep[Th. 11]{pillai2007}). Then, one cannot directly write $\langle x, K(s,\cdot)\rangle_K$, for $x$ a realization of the process. However, following \cite{parzen1961}, we will use the convenient notation 
 $\langle x, K(s,\cdot)\rangle_K$ interpreting this expression in terms of Lo\`eve's isometry, $\Psi_X$; more precisely, we will identify $\langle x, f\rangle_K$ with $\Psi_x^{-1}(f):= (\Psi_X^{-1}(f))(\omega)$, for $x=X(\omega)$ and $f\in \calh(K)$,  which in particular means $\Psi_X^{-1}(\sum_{i=1}^n a_i K(t_i,\cdot))=\sum_i a_i(X(t_i)-m(t_i))$. 
 
 The intuition behind the definition of $\langle x, f\rangle_K$  is reminiscent of the definition of It\^o's isometry,  which is used to define  the stochastic integral with respect to the Wiener measure (Brownian motion), overcoming the fact that the Brownian trajectories are not of bounded variation.  As we will see below, the transformation $x\mapsto \langle \beta,x\rangle_K=\Psi_X^{-1}(\beta)$ will play a central role in the alternative functional logistic model we are going to propose.

\

\noindent \textit{An RKHS-based proposal for logistic regression in the functional case} 

\

 We propose a new model for functional logistic regression problems, based on ideas borrowed from the theory of RKHS's. To be more specific, our proposal is to study the following model, instead of \eqref{Eq:logitL2},
\begin{equation}\label{eq:logistic}
\mathbb{P}(Y=1\, |\, X=x) \ = \ \frac{1}{1 + \exp\left\{-\beta_0 - \langle \beta,x \rangle_K \right\}},
\end{equation}
where the inner product stands for $\Psi_x^{-1}(\beta)$, the inverse of Lo\`eve's isometry defined in Equation \eqref{eq:isometry}. Throughout this paper we motivate this model and study some relevant theoretical aspects about it. 

\

\noindent \textit{Some RKHS related literature}

\

The book by \cite{hsing2015} provides an excellent mathematical background on mathematical methods, including RKHS theory, for the statistical analysis of functional data.
The papers by \cite{hsing2009} and 
\cite{kneip2020optimal} offer also very general perspectives and results on the applicability of RKHS methods in functional regression models, though not particularly focussed on the logistic case.

Some closely related ideas, aimed at the prediction problem in functional linear models are also present in \cite{shin2012linear}, even if the RKHS methodology is not explicitly mentioned there. 

Some other more specific references (a few of them especially dealing with the functional logistic model) will be cited below.

\newpage

\noindent \textit{The contents of this work}

\

In the first place (see Theorem \ref{Teo:LogModel}, Section \ref{sec:flrmodel}), we will analyze specific conditions under which the RKHS-based logistic model 
\eqref{eq:logistic} holds.

In the second place (see Theorem \ref{th:particular}, Section \ref{sec:particular}), we will show that model \eqref{eq:logistic}  covers some relevant cases of practical interest not included in the standard $L_2$-model \eqref{Eq:logitL2}, though, in fact, it is also shown in Theorem \ref{th:particular} that the $L_2$ model can also be obtained as a particular case of the RKHS model \eqref{eq:logistic} under some conditions. 

In the third place, in Section \ref{sec:MLE}, we will prove two results of non-existence for the maximum likelihood estimator of the slope function $\beta$ in model \eqref{eq:logistic}. Such negative results can be seen as an aggravated, functional counterpart of the well-known partial non-existence results arising in finite dimensional logistic  models; see \cite{candes2020phase} and references therein. 

This is a methodological and theoretical paper: our aim is to contribute to a better understanding of the functional logistic model, rather than focussing on numerical or practical issues. However, we provide in Section \ref{sec:exPr} an specific suggestion (based on a restricted maximum likelihod approach) to deal with the estimation of the slope function in model \ref{eq:logistic}.

\section{The RKHS-based functional logistic model: validity conditions in the Gaussian case}\label{sec:flrmodel}

In this section we motivate the reasons why model \eqref{eq:logistic} is meaningful.  In Theorem \ref{Teo:LogModel} we show that the standard assumption that both $X|Y=0$ and $X|Y=1$ are Gaussian implies \eqref{eq:logistic}.  We also analyze under which conditions  the more standard $L^2$-model  \eqref{Eq:logitL2} is implied and we clarify the difference between both approaches.


In  our   functional setting, for $i=0,1$, we assume that $\{X(t):\, t\in[0,1]\}$ given $Y=i$ is a Gaussian process with  continuous mean function $m_i$ and continuous covariance function $K$ (the same for $i=0,1$). We will assume throughout that all the eigenvalues $\lambda_i$ of the covariance operator ${\mathcal K}$, associated with $K$ are strictly positive (so ${\mathcal K}$ is injective). Note that, as a consequence of Spectral Theorem (see, e.g., \cite[p. 98]{hsing2015}) ${\mathcal K}x=\sum_j\lambda_i\langle x,e_i\rangle e_i)$, where $e_i$ stands for a unit eigenvector associated with $\lambda_i$; thus, the inverse ${\mathcal K}^{-1}$ is defined on the range of ${\mathcal K}$, ${\mathcal K}(L^2)$, as a linear (not continuous) transformation, by ${\mathcal K}^{-1}y=\sum_i\frac{\langle y,e_i\rangle}{\lambda_i}e_i)$, for $y=\sum_i\langle y,e_i\rangle e_i\in {\mathcal K}(L^2)$. 

Let $P_{m_0}$ and $P_{m_1}$ be the  probability measures (i.e., the distributions) induced  by the process $X$ conditional to $Y=0$ and $Y=1$ respectively. Recall that when $m_0$ and $m_1$ both belong to $\calh(K)$, we have that $P_{m_0}$ and $P_{m_1}$ are mutually absolutely continuous; see Theorem 5A of \cite{parzen1961}. The following theorem provides a very natural motivation for the RKHS model \eqref{eq:logistic} in this Gaussian setting.

\begin{theorem}\label{Teo:LogModel}
Let $P_{m_0}, P_{m_1}$  and ${\mathcal K}$  be as in the previous lines. Then,
\begin{itemize}
\item[(a)] if $m_1-m_0\in\mathcal{H}(K)$, then $P_{m_0}$ and $P_{m_1}$ are mutually absolutely continuous and  model \eqref{eq:logistic} holds, 
$$\mathbb{P}(Y=1\, |\, X=x) \ = \ \frac{1}{1 + \exp\left\{-\beta_0 - \langle x,\beta \rangle_K \right\}} \ \equiv \ \frac{1}{1 + \exp\left\{-\beta_0 - \Psi_x^{-1}( \beta) \right\}},$$
with $\beta := m_1-m_0$ and 
$\beta_0 := - \mathbb{E}_{m_1}[\Psi_x^{-1}( \beta)] + \|m_1-m_0\|^2_{K}/2-\log((1-p)/p)$ (where $p=\mathbb{P}(Y=1)$ and 
$\mathbb{E}_{m_1}(\cdot)$ stands for the expectation when the  process has mean function equal to $m_1$). If $m_1-m_0\notin\mathcal{H}(K)$, then $P_{m_0}$ and $P_{m_1}$ are mutually singular.
\item[(b)] if $m_1-m_0\in\mathcal{K}(L^2) = \{\mathcal{K}(f) \, : \, f\in L^2[0,1]\}$, then $P_{m_0}$ and $P_{m_1}$ are mutually absolutely continuous and model \eqref{Eq:logitL2} holds.
\item[(c)] if $m_1-m_0\not\in \mathcal{K}(L^2)$ model \eqref{Eq:logitL2} is never recovered, but different situations are possible, according to the condition in part (a). In particular if $m_0=0$, $m_1\in\mathcal{H}(K)$ recovers scenario (a), but if $m_1\not\in\mathcal{H}(K)$, $P_{m_0}$ and $P_{m_1}$ are mutually singular. 
\end{itemize}
\end{theorem}

\begin{proof}
	(a) Let $P_0$ be the measure induced by a Gaussian process with covariance function $K$ but zero mean function, $m\equiv 0$. From Theorem 7A in \cite{parzen1961approach} $m_0-m_1\in \mathcal{H}(K)$ implies that $P_{m_0-m_1}$ and $P_0$ are mutually absolutely continuous, and $m_0-m_1\notin \mathcal{H}(K)$ implies that $P_{m_0-m_1}$ and $P_0$ are mutually singular. By Lemma 1.1 in \cite{pitcher1960likelihood},  $P_{m_0-m_1}$ and $P_0$ are mutually absolutely continuous if and only if $P_{m_0}$ and $P_{m_1}$ are mutually absolutely continuous and, in this case, the corresponding Radon-Nikodym derivatives fulfill  
	\[
	\frac{\dd P_{m_0}}{\dd P_{m_1}}(X) = \frac{\dd P_{m_0-m_1}}{\dd P_{0}}(X-m_1) = 
	\exp\left\{\langle X-m_1,m_0-m_1\rangle_{K} - \frac{1}{2}\|m_0-m_1\|^2_{K}\right\}.
	\]
	The last equality also follows from Theorem 7A in \cite{parzen1961approach} (or Theorem~5A of \cite{parzen1961}). Notice that by the definition of Loève's isometry we have $\langle X-m_1,m_0-m_1\rangle_{K} = \langle X,m_0-m_1\rangle_{K} - \mathbb{E}_{m_1}[\langle X,m_0-m_1\rangle_{K}]$.

	The conditional probability of $Y=1$ can be expressed in terms of the Radon-Nikodym derivative of $P_1$ with respect to $P_0$ (see \citet[Th.1]{baillo2011}) by
	\begin{equation}\label{Eq:logRadon}
	\mathbb{P}(Y=1\, |\, X) \ = \ \frac{p\frac{\dd P_{m_1}}{\dd P_{m_0}}(X)}{p\frac{\dd P_{m_1}}{\dd P_{m_0}}(X) + (1-p)} \ = \ \left(1+\frac{1-p}{p}\frac{\dd P_{m_0}}{\dd P_{m_1}}(X) \right)^{-1}.
	\end{equation}
	
	From the last two displayed equations, one can rewrite
	\[
	\mathbb{P}(Y=1\, |\, X) = \left(1 + \frac{1-p}{p} \exp\left\{\langle X,m_0-m_1\rangle_{K} - \mathbb{E}_{m_1}[\langle X,m_0-m_1\rangle_{K}]  - \frac{1}{2}\|m_0-m_1\|^2_{K}\right\}  \right)^{-1}.
	\]
	Then, reordering terms in this expression we get the logistic model in part (a).
	
	\
	
	(b) Under the assumptions, Theorem 6.1 in \cite{rao1963} gives the following expression:
	\begin{equation*}
	\label{eq:RN-L2}
	\log\Big(\frac{\dd P_{m_1}}{\dd P_{m_0}}(x)\Big) = \langle x-m_0,\, \mathcal{K}^{-1}(m_1-m_0)\rangle_2 - \frac{1}{2} \, \langle m_1-m_0,\, \mathcal{K}^{-1}(m_1-m_0)\rangle_2,  
	\end{equation*}
	for $x\in L^2[0,1]$.  This entails (using the Chain Rule for Radon-Nikodym derivatives)
	$$
	\frac{\dd P_{m_0}}{\dd P_{m_1}}(x)=C\exp \left(-\langle x,\beta \rangle_2
	\right),
	$$
	where $\beta={\mathcal K}^{-1}(m_1-m_0)$ and $C=\exp(\langle m_0+m_1,\beta\rangle_2/2$. 
	Now, replacing this expression in \eqref{Eq:logRadon} 
	we get the $L^2$-model \eqref{Eq:logitL2} with $\beta_0=-\log\big(\frac{1-p}{p}C\big)$.

	\
	
	(c) Also as a consequence of Theorem 6.1 in \cite{rao1963}, if $m_1-m_0 \notin \operK(L^2)$ it is not possible to express the Radon-Nikodym derivative in terms of inner products in $L_2$ or, equivalently, there is not any continuous linear functional $L(x)$ and $c\in\mathbb{R}$ such that $\log(\frac{\dd P_1}{\dd P_0}(x)) = L(x) + c$.  Finally, the last sentence of the statement is a consequence of Theorem 5A of \cite{parzen1961}.
\end{proof}

\noindent \textit{Some comments on the meaning of  Theorem \ref{Teo:LogModel} }

\

		Similarly to the finite-dimensional case, our model holds when the conditional distributions of the process given the two possible values of $Y$ are Gaussian with the same covariance structure. Another interesting property of this new  RKHS-based  model is that for some particular choices of the slope function of type $\beta(\cdot)=\sum_i^pa_iK(t_i,\cdot)$, the model \eqref{eq:logistic} amounts to a finite-dimensional logistic regression model for which the  explanatory variables  are a finite number of projections of the trajectories of the process. Thus, the impact-point model studied by \cite{lindquist2009}  appears as  a particular case of the RKHS-based model  and, more generally, model \eqref{eq:logistic} can be seen as a true extension of the finite-dimensional logistic regression model, which is obtained when a finite-dimensional covariance matrix plays the role of the kernel. As an important by-product, this provides a mathematical ground for variable selection in logistic regression.

	Part (b) of this theorem has been recently observed by \cite{petrovich2019highly}, see Theorem~1, without reference to RKHS theory. Note that, in general,   $m_1-m_0 \notin \operK(L^2)$ does not imply that $P_{m_1}$ and $P_{m-0}$ are orthogonal. Parts (a) and (c) of  the theorem above clarifies this point. 
	
	On the other hand, in order to better interpret the above theorem in RKHS terms, let us recall that  the space $\calh(K)$ can be also defined as the image of the square root of the covariance operator defined in \eqref{Eq:operK} (e.g. Definition 7.2 of \cite{peszat2007}),
	$$\calh(K) = \{ \operK^{1/2}(f), \ f\in L^2[0,1]\},$$
	where now the inner product is defined, for $f,g \in \calh(K)$, as
	$$\langle f, g \rangle_K = \langle \operK^{-1/2}(f), \ \operK^{-1/2}(g) \rangle_2.$$
	It can be seen that this definition of $\calh(K)$ is equivalent to that given in Section \ref{sec:int}. Then, from part (c) of the theorem it follows that the RKHS functional logistic regression can be seen as a generalization of the usual $L_2$ functional logistic regression  model, in the sense that this $L_2$ model  is recovered when a higher degree of smoothness on the mean functions is imposed (since clearly $\operK(L^2)\subsetneq \calh(K)$). Indeed, the functions in $\operK(L^2)$ are convolutions of the functions in $L^2[0,1]$ with the covariance function of the process. The discussion of the next section makes clear that this difference is of key importance in practice and not merely a technicality.
	
As mentioned above, in the finite dimensional case the logistic model holds whenever $X|Y=i$ are Gaussian and homoscedastic, but in fact this model is more general in the sense that it also holds for other non-Gaussian assumptions on the conditional distributions $X|Y=i$.  Clearly this is also the case for the functional logistic model  \eqref{eq:logistic}.  In fact, the connection between the functional model and the finite-dimensional one is even deeper, as we will show in the following section.


\section{The RKHS model: some important particular cases} \label{sec:particular}

Dimension reduction in the functional logistic regression model may be often appropriate in terms of interpretability of the model and classification accuracy. This reduction must be done losing as little information as possible. We propose to perform variable selection on the curves. 
By variable selection we mean to replace each curve $x_i$ by the finite-dimensional vector $(x_i(t_1),\ldots,x_i(t_p))$, for some $t_1,\ldots,t_p$  chosen in an optimal way. In this section we analyze under which conditions it is possible to perform functional variable selection, which is only feasible under the RKHS-model. In the following section we  suggest  how to do it:  the idea is incorporating  the points $t_1,\ldots,t_p$ to the estimation procedure as additional parameters (in particular to the modified maximum likelihood estimator we propose).

Whenever the slope function $\beta$ has the form 
\begin{equation}\label{Eq:betafinLog}
\beta(\cdot) = \sum_{j=1}^p \beta_j K(t_j,\cdot),
\end{equation}
the model in \eqref{eq:logistic} is reduced to the finite-dimensional one,
\begin{equation}\label{eq:logfin}
\P(Y=1 | X) = \bigg(1+\exp\Big\{-\beta_0 - \sum_{j=1}^p \beta_j(X(t_j)-m(t_j))\Big\} \bigg)^{-1}.
\end{equation}
The main difference between the standard finite-dimensional model and this one is that now the proper choice of the points $T=(t_1,\ldots,t_p)\in[0,1]^p$ is a part of the estimation procedure.  In this sense, model \eqref{eq:logfin} is truly functional since we will use the whole trajectories $x_i(t)$ to select the points.   This fact leads to a critical difference between the functional and  the  multivariate problems. Then, our aim is to approximate the general model described by Equation~\eqref{eq:logistic} with finite-dimensional models as those of Equation~\eqref{eq:logfin}. This amounts to get an approximation of the slope function in terms of a finite linear combination of kernel evaluations $K(t_j,\cdot)$. This model, for $p=1$ and a particular type of Gaussian process $X$, is analyzed in \cite{lindquist2009}.

From the discussion above, it is clear that the differences between the RKHS model and the $L^2$ one are not minor technical questions. The functions of type $\beta(\cdot) = K(\cdot,t)$ belong to $\calh(K)$ but do not belong to $\operK(L^2)$. This fact implies that within the setting of the RKHS model it is possible to regress $Y$ on any finite dimensional projection of $X$, whereas  this  does not make sense if we consider the $L^2$ model. This feature is clearly relevant if one wishes to analyze properties of variable selection methods.  

\begin{theorem}\label{th:particular}
	Assume model (\ref{eq:logistic}) holds. Then,
	\begin{enumerate}
		\item[(a)] If there exist a positive integer $p$, $\beta_1,\ldots,\beta_p \in \mathbb{R}$, and $t_1,\ldots,t_p \in [0,1]$ such that $\beta(\cdot)=\sum_{j=1}^p \beta_j K(\cdot,t_j)$, then
		\[
		\mathbb{P}(Y=1\, |\, X=x) =  \frac{1}{1 + \exp\left\{-\beta_0 - \sum_{j=1}^p \beta_j (x(t_j)-m(t_j)) \right\}}
		\]
		\item[(b)] If $\beta\in \mathcal{K}(L^2)=\{\mathcal{K}(f):\, f\in L^2[0,1]\}$, then
		\[
		\mathbb{P}(Y=1\, |\, X=x) =  \frac{1}{1 + \exp\left\{-\beta_0 - \int_0^1 \alpha(t)(x(t)-m(t))dt \right\}},
		\]
		where $\alpha\in L^2[0,1]$ fulfills  $\beta = \mathcal{K}(\alpha)$.
		\item[(c)] Let $\{u_j\}$ be an orthonormal basis of $L^2[0,1]$. If there exist a positive integer $p$, and $\beta_1,\ldots,\beta_p \in \mathbb{R}$  such that $\beta = \sum_{j=1}^p \beta_j \mathcal{K}(u_j)$, then
		\[
		\mathbb{P}(Y=1\, |\, X=x) =  \frac{1}{1 + \exp\left\{-\beta_0 - \sum_{j=1}^p \beta_j\langle x-m,u_j\rangle_2 \right\}},
		\]
	\end{enumerate}
\end{theorem}

\begin{proof}
	(a) Observe that for $j=1,\ldots,p$, $X(t_j)-m(t_j)\in \mathcal{L}_0(K)$, and for all $s\in[0,1]$,
	\[
	\Psi_X(X(t_j)-m(t_j))(s) = \mathbb{E}[(X(s)-m(s))(X(t_j)-m(t_j))] = K(s,t_j).
	\]
	Therefore $\Psi_X^{-1}(k(\cdot, t_j)) = X(t_j)-m(t_j)$, and
	\[
	\langle X,\beta\rangle_K  = \sum_{j=1}^p \beta_j \langle X,k(\cdot,t_j)\rangle_K = \sum_{j=1}^p \beta_j(X(t_j)-m(t_j)).
	\]
	
	(b) Let $U :=  \int_0^1 \alpha(t)(X(t)-m(t))dt$. It holds that $U\in \mathcal{L}(X)$ (see e.g. \cite{ash2014topics}, page 34). Moreover,  by Fubini's theorem, for all $s\in[0,1]$ we have
	\begin{equation}
	\label{eq:models}
	\Psi_X(U)(s) = \mathbb{E}\left[\int_0^1 \alpha(t)(X(t)-m(t))dt \cdot (X(s) - m(s))\right] = \int_0^1 K(s,t)\alpha(t)dt = \beta(s),
	\end{equation}
	because $\mathcal{K}(\alpha)=\beta$. Therefore, $\langle X,\beta\rangle_K = U = \int_0^1 \alpha(t)(X(t)-m(t))dt$.
	
	(c) Putting $\alpha(t) = u_j(t)$ in (\ref{eq:models}) we get $\Psi_X(\langle u_j,X-m\rangle_2) = \mathcal{K}(u_j)$. As a consequence,
	\[
	\langle X,\beta\rangle_K  = \sum_{j=1}^p \beta_j \langle u_j,X-m\rangle_2.
	\]
\end{proof}

\noindent\textit{Some comments on the meaning of Theorem \ref{th:particular} }

\

Part (a) of the previous result means that the
impact point  model, as that considered in \cite{lindquist2009}, is a particular case of the RKHS model \eqref{eq:logistic}. Just take as parameter function $\beta$ a finite linear combination of evaluations of $K$.

Part (b) implies that the usual functional model based on the $L^2$ inner product is also a particular case of (\ref{eq:logistic}). What we need is that $\beta$ belongs to the image of the covariance operator, $\mathcal{K}(L^2)$. Notice that this condition is stronger than $\beta\in \mathcal{H}(K) = \mathcal{K}^{1/2}(L^2)$. As illustrated by part (a), the difference between  $\mathcal{K}^{1/2}(L^2)$ and $\mathcal{K}(L^2)$ may be important in practice. 

A very common methodology to fit functional regression models requires to project the functional regressors on a subspace defined by a finite set of orthonormal functions $u_1,\ldots, u_p$, and use the projections as regressor variables. Part (c) implies that model (\ref{eq:logistic}) also includes this situation for $\beta$ in the span of $\mathcal{K}(u_1),\ldots,\mathcal{K}(u_p)$. Note that if $\{u_j\}$ is the orthonormal basis of eigenfunctions of $\mathcal{K}$, we have $\mathcal{K}(u_j)$ is proportional to $u_j$, and the condition on $\beta$ reduces to the fact that $\beta$ belongs to the span of $u_1,\ldots,u_p$. If this is the case, there is no loss in using the first $p$ principal components of the regressors instead of the whole trajectories.

\section{Maximum likelihood estimation: non-existence results}\label{sec:MLE}

In the finite-dimensional setting, it is well-known that the  maximum likelihood (ML)  estimator does not exist when there is  an hyperplane separating the observations of the two classes; see below for details. As we will show in this section, this fact worsens dramatically for the case of  functional data; more specifically, we will see that:
\begin{itemize}
	\item[] For a wide class of process (including the Brownian motion),  the MLE just does not exist, with probability one (see Subsection \ref{subsec:NEML}).  
	\item[] Under some different conditions, in the Gaussian case, the probability of non-existence of the MLE tends to one  when the sample size tends to infinity (see Subsection \ref{subsec:asymptoticNE}).
	
\end{itemize}

\noindent \textit{A brief overview of the finite dimensional case}

\

Despite the fact that ML estimation of the slope function for multiple logistic regression is widely used, it has an important  drawback  that is sometimes overlooked. Given a sample $x_i^0\in \R^d$ for $i=1,\ldots,n_0$ drawn from population zero and another sample $x_i^1\in \R^d$ for $i=1,\ldots,n_1$ drawn from population one, the classical MLE in logistic regression is the vector $(b_0, b)\in\R\times\R^d$ that maximizes the log-likelihood
$$L_n(b,b_0) = \frac{1}{n_0} \sum_{i=1}^{n_0} \log\Big(\frac{e^{-b_0-b'x_i^0}}{1+e^{-b_0-b'x_i^0}} \Big) + \frac{1}{n_1} \sum_{i=1}^{n_1} \log\Big(\frac{1}{1+e^{-b_0-b'x_i^1}} \Big).$$

The existence and uniqueness of such a maximum was carefully studied by \cite{albert1984} (and previously by \cite{silvapulle1981} and \cite{gourieroux1981}). As stated in Theorem 1 of \cite{albert1984}, the latter expression can be made arbitrarily close to zero (note that the log-likelihood is always negative) whenever the samples of the two populations are linearly separable. In that case the maximum can not be attained and then the MLE does not exist (the idea behind the proof is similar to the one of Theorem \ref{Teo:ne} below). There is another scenario where this estimator does not exist; the samples are linearly separable except for some points of both populations that fall into the separation hyperplane (named ``quasicomplete separation''). In this case the supremum of the log-likelihood function is strictly less than zero, but it is anyway unattainable.

\

\noindent \textit{The likelihood function in the logistic functional model}

\

Before going on with the functional case (which is our main target here), we need to derive the likelihood function. Let assume that $\{X(s), s\in[0,1]\}$ follows the RKHS logistic model described in Equation \eqref{eq:logistic}. That is, 
$$\beta_0+\Psi_X^{-1} (\beta) \ \equiv \ \beta_0+\langle X, \beta \rangle_K \ = \ \log \Big(\frac{p_{\beta,\beta_0}(X)}{1-p_{\beta,\beta_0}(X)}\Big),$$
where $p_{\beta,\beta_0}(X) = \P (Y=1 | X,\beta,\beta_0)$, $\beta_0\in\R$ and $\beta \in \mathcal{H}(K)$. The random element $(X(\cdot),Y)$ takes  values in the space $Z=L^2[0,1]\times \{0,1\}$, which is a measurable space with measure $z = P_X\times \mu$, where $P_X$ is the distribution induced by the process $X$ and $\mu$ is the counting measure on $\{0,1\}$. We can define in $Z$ the measure $P_{(X,Y);\beta,\beta_0}$, the joint probability induced by $(X(\cdot),Y)$ for a given slope function $\beta$ and an intercept $\beta_0$.  Then we define, 
\begin{eqnarray*}
f_{\beta,\beta_0}(x,y) & = & \frac{\mathrm{d} P_{(X,Y);\beta,\beta_0}}{\mathrm{d} z}(x,y) =  p_{\beta,\beta_0}(x)^{y}\left(1-p_{\beta,\beta_0}(x)\right)^{1-y}\\
&=&\Big(\frac{1}{1+e^{-\beta_0-\langle \beta, x\rangle_K}}\Big)^y \Big(\frac{e^{-\beta_0-\langle \beta, x\rangle_K}}{1+e^{-\beta_0-\langle \beta, x\rangle_K}}\Big)^{1-y}.
\end{eqnarray*}
In view of this density function, the log-likelihood function for a given sample in $L^2[0,1]\times\{0,1\}$ is 
$$L_n(\beta,\beta_0) \ = \ \frac{1}{n} \sum_{i=1}^n \log\big(p_{\beta,\beta_0}(x_i)^{y_i}\left(1-p_{\beta,\beta_0}(x_i)\right)^{1-y_i}\big),$$
where $(x_i,y_i)\in L^2[0,1]\times \{0,1\}$ is a sample of the underlying random variable $(X,Y)$.

The maximum likelihood estimator is the pair $(\wh\beta,\wh\beta_0)$ that maximizes this function $L_n$.  The population counterpart of $L_n$ is the expected log-likelihood function,
\begin{equation}\label{Eq:flrloglike} 
L(\beta,\beta_0) \ = \ \E_Z \left[ \log f_{\beta,\beta_0}(X,Y) \right] \ = \ \E_Z\left[ \log\left(p_{\beta,\beta_0}(X)^{Y}\left(1-p_{\beta,\beta_0}(X)\right)^{1-Y}\right)\right],
\end{equation}
 where $\E_Z[\cdot]$ denotes the expectation with respect to the measure $\dd z$.  
 
 The main idea behind ML-estimation stands in the infinite-dimensional situation. If our ``parameter space'' is $\Theta\subset {\mathcal H}(K)\times {\mathbb R}$ and the ``true'' value of the parameter is $(\beta^*,\beta_0^*)\in\Theta$, then a simple, standard argument based on 
 Jensen's inequality shows that the population log likelihood function $L(\beta,\beta_0)$ fulfils 
 $$
 L(\beta^*,\beta_0^*)\geq L(\beta,\beta_0),\ \ \mbox{for all}\ \ (\beta,\beta_0)\in\Theta.
 $$
 This leads to the usual, natural idea of maximizing a consistent estimator of $L(\beta^*,\beta_0^*)$ that, in our logistic model, is the log-likelihood function
 $L_n(\beta,\beta_0)$ defined above.

\

\subsection{Non-existence of the MLE in functional settings}\label{subsec:NEML}

We first show that, when moving from the finite-dimensional model to the functional one, the problem of the non-existence of the MLE is drastically worsened. 

This situation is quite similar to that arising, for example, in non-parametric density estimation where non-parametric (and non-penalized) ML estimators of the density function do not exist, unless some drastic restrictions, such as monotonicity (e.g., \cite{grenander1981}) or log-concavity (see, \cite{cule2010}) are imposed on the underlying density function.

Since the analogous non-existence result for the case of the functional logistic regression model is not perhaps so direct, it is established in Theorem \ref{Teo:ne} below. We confine ourselves to the  RKHS-based model \eqref{eq:logistic}, although the result can be easily extended, with a completely similar method of proof, for the standard $L^2$ based model of Equation \eqref{Eq:logitL2}.

We first will need to establish a condition which plays, in the functional case, a similar role to that of the linear separability condition mentioned above in the setting of finite-dimensional logistic regression. 

\begin{assumption}[\textbf{SC}]
The multivariate process $Z(t)=(X_1(t),\ldots, X_n(t))$, $t\in[0,1]$ satisfies the ``Sign Choice'' (SC) property when for all possible choice of signs $(s_1,\ldots,s_n)$, where $s_j$ is either $+$ or $-$, we have that, with probability one, there exists some $t_0\in[0,1]$ such that $\mbox{sign}(X_1(t_0))=s_1,\ \ldots, \ \mbox{sign}(X_n(t_0))=s_n$. 
\end{assumption}

Now, the non-existence result is as follows. Without loss of generality we confine ourselves to the case ${\mathbb E}(X(t))=0$. 

\begin{theorem}\label{Teo:ne}
Let $X(s)$, $s\in[0,1]$, be an $L^2$ stochastic process with $\E[X(s)]=0$. Denote by $K$ the corresponding covariance function. Consider a logistic model \eqref{eq:logistic} based on $X(s)$. Let $X_1,\ldots,X_n$ be independent copies of $X$. Assume that the $n$-dimensional process $Z_n(s)=(X_1(s),\ldots,X_n(s))$ fulfills the SC property. Then, with probability one,  the MLE estimator of $(\beta,\beta_0)$ (i.e., the maximizer of the log-likelihood function $L_n(\beta,\beta_0)$)  does not exist for any sample size $n$.
\end{theorem}

\begin{proof}
Let $x_1(s)\ldots,x_n(s)$ be a random sample drawn from $X(s)$. From the SC assumption there is (with probability 1) one point $t_0$  such that $x_i(t_0)>0$ for all $i$ such that $y_i=1$ and $x_i(t_0)<0$ for those indices $i$ with $y_i=0$.  Note that  that the sample log-likelihood function can be split in two terms, as follows,
$$L_n(\beta,\beta_0) \ = \frac{1}{n}\sum_{\{i:\,y_i=1\}} \log \Big(\frac{1}{1+e^{-\beta_0-\langle \beta,x_i\rangle_K}}\Big) + \frac{1}{n}\sum_{\{i:\,y_i=0\}} \log \Big(\frac{e^{-\beta_0-\langle \beta,x_i\rangle_K}}{1+e^{-\beta_0-\langle \beta,x_i\rangle_K}}\Big).$$
 Note also that $L_n(\beta,\beta_0)\leq 0$ for all $\beta$. Now, take a numerical sequence $0<c_m \uparrow \infty$ and define 
$$\beta_m(\cdot) \ = \ c_m K(t_0,\cdot).$$
 Then, by the definition of Loève's isometry, if $y_i=0$,
$$ \langle \beta_m,x_i\rangle_K \ = \ c_m x_i(t_0) \ \to \ \infty,\ \text{ as } m\to\infty, $$
since we have taken $t_0$ such that $x_i(t_0)>0$ for those indices $i$ with $y_i=1$. Likewise, $\langle \beta_m,x_i\rangle_K$ goes to $-\infty$ whenever $y_i=0$ since we have chosen $t_0$ such that $x_i(t_0)<0$ for those indices. As a consequence $L_n(\beta_m,0)\to 0$ as $m\to \infty$. Therefore the likelihood function can be made arbitrarily large so that the MLE does not exist. 
\end{proof} 

\begin{remark}
A non-existence result for the MLE estimator, analogous to that of Theorem~\ref{Teo:ne}, can be also obtained with a very similar reasoning for the $L^2$-based logistic model of Equation~\eqref{Eq:logitL2}. The main difference in the proof would be the construction of $\beta_m$ which, in the $L^2$ case, should be obtained as an approximation to the identity (that is, a linear ``quasi Dirac delta'') centered at the point $t_0$. 
\end{remark}

Although  the SC property  could seem a somewhat restrictive assumption,  the following proposition shows that   it  applies to some important and non-trivial situations. 

\begin{proposition}\label{ex}
	(a) The $n$-dimensional Brownian motion fulfills the SC property.
	
	(b) The same holds for any other $n$-dimensional process in $[0,1]$ whose independent marginals have a distribution absolutely continuous with respect to that of the Brownian motion. 
\end{proposition}

\begin{proof}
	(a) Given the $n$ dimensional Brownian motion $\calb_n=(B_1,\ldots,B_n)$, where the $B_j$ are independent copies of the standard Brownian motion $B(t)$, $t\in[0,1]$, take a sequence of signs $(s_1,...,s_n)$ and define the event
	\begin{equation}
	A=\{\mbox{for any given } t \mbox{ there exists } 0<t_0<t \mbox{ s.t. } \mbox{sign}(B_j(t_0))=s_j,\ j=1,\ldots,n\}\label{A}
	\end{equation}
	We may express this event by
	\begin{equation}
	A=\bigcap_{t\in (0,1]\cap\mathbb{Q}} A_t,\label{At}
	\end{equation}
	where, for each $t\in(0,1]\cap\mathbb{Q}$, 
	$$
	A_t=\{\mbox{there exists } t_0<t \mbox{ such that } \mbox{sign}(B_j(t_0))=s_j,\ j=1,\ldots,n\}.
	$$
	Now, the result follows directly from Blumenthal's 0-1 Law for n-dimensional Brownian processes (see, e.g., \citet[p. 38]{morters2010}). Such result establishes that for any event $A\in{\mathcal F}^+(0)$ we have either ${\mathbb P}(A)=0$ or ${\mathbb P}(A)=1$. Here ${\mathcal F}^+(0)$ denotes the \textit{germ $\sigma$-algebra} of events depending only on the values of $\calb_n(t)$ where $t$  lies  in an arbitrarily small interval on the right of 0. More precisely,
	$$
	{\mathcal F}^+(0)=\bigcap_{t>0} {\mathcal F}^0(t),\ \mbox{where } {\mathcal F}^0(t)=\sigma(\calb_n(s),0\leq s\leq t).
	$$ 
	From \eqref{A} and \eqref{At} it is clear that the above defined event $A$ belongs to the germ $\sigma$-algebra ${\mathcal F}^+(0)$. However, we cannot have ${\mathbb P}(A)=0$ since (from the symmetry of the Brownian motion) for any given $t_0$ the probability of $\mbox{sign}(B_j(t_0))=s_j,\ j=1,\ldots,n$ is $1/2^n$. So, we conclude 
	${\mathbb P}(A)=1$ as desired.

	(b) If $X(t)$ is another process whose distribution is absolutely continuous with respect to that of the n-dimensional Brownian motion $\calb_n$, then the set $A$, defined by \eqref{A} and \eqref{At} in terms of $\calb_n$ has also probability one when it is defined in terms of the process $X(t)$: recall  that, from the definition of absolute continuity, if the set $A^c$ has  probability zero under the Brownian motion, then its probability must be zero as well when $B(t)$ is replaced with $X(t)$. Therefore, the probability of $A$ under $X=X(t)$ must be one. 
\end{proof}

\begin{remark}
Following the comment in \cite{morters2010} about processes with strong Markov property, this result based on RKHS theory can be extended for L\'evy processes whenever the covariance function was continuous (like Poisson process in the real line). However note that, apart from the Brownian motion, this type of processes have discontinuous trajectories.
\end{remark}

 The situation considered in Theorem \ref{Teo:ne}  would be the functional counterpart of having a finite-dimensional problem where the supports of both classes (0 and 1) are linearly separable. However,  as we have just seen, this separability issue does not only appear in degenerate problems in the functional setting. In the next section we suggest a technique to completely avoid the problem.

From a theoretical perspective, in view of Theorem \ref{Teo:ne}, it is clear that there is no hope of obtaining a general convergence result of the standard  maximum likelihood estimator (MLE)  defined by the maximization of the likelihood function $L_n(\beta,\beta_0)$. That is, one should define a different estimator or impose some conditions on the process $X$ to avoid the SC property. For instance, \cite{lindquist2009} prove consistency results of the model with a single impact point $\theta\in [0,1]$ for processes $X(t) = Z + B_\theta(t)$, where $B_\theta$ is a two-sided Brownian motion centered in $\theta$ (i.e. two independent Brownian motions starting at $\theta$ and running in opposite directions) and $Z$ is a real random variable independent of $B_\theta$. Then, due to the independence assumption, it is clear that accumulation points (like 0 for the Brownian motion) are avoided.

\subsection{Asymptotic non-existence for Gaussian processes}\label{subsec:asymptoticNE}

In the previous section we have seen that the problem of non-existence of the  MLE is aggravated for the case functional data. But this is not the only issue with MLE in functional logistic regression. In this section we see that the probability that the MLE does not exist goes to one as the sample size increases, for any Gaussian process satisfying very mild assumptions.

We use the following notation: for $T=\{t_1,\ldots,t_p\}\subset [0,1]$ and $f\in L^2[0,1]$, let $f(T):=(f(t_1),\ldots,f(t_p))'$ and let $\Sigma_T$ be the $p\times p$ matrix whose $(i,j)$ entry is $K(t_i,t_j)$. 

\begin{theorem}\label{Teo:neasym}
Let $(x_1,y_1),\ldots,(x_n,y_n)$ be a random sample of independent observations satisfying model (\ref{eq:logistic}). Assume that $X$ is a Gaussian process such that $K$ is continuous and $\Sigma_T$ is invertible for any finite set $T\subset (0,1)$. It holds 
\[
\lim_{n\to\infty} \P(\mbox{MLE exists})=0. 
\]
\end{theorem}

\begin{proof}
Let $\beta^*\in\mathcal{H}_K,\beta_0^*$ be the true values of the parameters. Since $\|\beta^*\|_K<\infty$, we have $h(\beta_0^*, \|\beta^*\|_K)<\infty$, where $h$ is the function defined in \cite{candes2020phase}, Equation (2.2) (see  Remark \ref{remark:candes} below). Let $p_n$ be an increasing sequence of natural numbers such that $\lim_{n\to\infty} p_n/n =\kappa > h(\beta_0^*, \|\beta^*\|_K)$.  Consider the set of equispaced points $0<t_1<t_2<\cdots<t_{p_n}<1$ and denote $T_n=\{t_1,\ldots,t_{p_n}\}$. Define $\alpha_{T_n} = \Sigma_{T_n}^{-1}\beta^*(T_n)$. Now, consider the following sequence of finite-dimensional logistic regression models
\[
\mathbb{P}(Y=1\, |\, X) = \frac{1}{1 + \exp\left\{-\beta_0^* - \alpha_{T_n}'X(T_n)  \right\}},
\]
where $a'b$ stands here for the inner product of two vectors $a,b$ in ${\mathbb R}^{p_n}$, and the following sequence of events
\[
E_n = \{\mbox{There exists}\ \alpha\in\mathbb{R}^{p_n}:\, \alpha'x_i(T_n)\geq 0,\ \mbox{if}\ y_i=1;\ \ \alpha'x_i(T_n)\leq 0,\ \mbox{if}\ y_i=0\}.
\]
Recall that the event $E_n$ amounts to non-existence of MLE for finite-dimensional logistic regression models (see \cite{albert1984}).

Now let us prove the validity of condition (1.3) in \cite{candes2020phase}, which is required for the validity of Theorem 2.1. in that paper. In our case, such condition amounts to 
\[
\lim_{n\to\infty} \var\big(\alpha_{T_n}'X(T_n)\big) = \lim_{n\to\infty} \alpha_{T_n}'\Sigma_{T_n}\alpha_{T_n}=\|\beta^*\|_K^2,
\]
but this directly follows from Theorem 6E of \cite{parzen1959}. Since $\lim_{n\to\infty} p_n/n =\kappa > h(\beta_0^*, \|\beta^*\|^2_K)$ we apply Theorem 2.1. in  \cite{candes2020phase} to  get  $\lim_{n} \mathbb{P}(E_n)=1$.

Now we define the auxiliary sequence of events
 \[
\widetilde{E}_n = \{\mbox{There exists}\ \alpha\in\mathbb{R}^{p_n}:\, \alpha'x_i(T_n)> 0,\ \mbox{if}\ y_i=1;\ \ \alpha'x_i(T_n)< 0,\ \mbox{if}\ y_i=0\},
\]
with strict inequalities. Assume that $\widetilde{E}_n$ happens so that there exists a separating hyperplane defined by $\alpha\in\mathbb{R}^{p_n}$. Then, in the same spirit as in the proof of Theorem \ref{Teo:ne}, it is possible to show that if $\hat{\beta}_{m,n}=m\sum_{j=1}^{p_n}\alpha_jK(\cdot,t_j)\in\mathcal{H}_K$, then  $\lim_{m\to\infty} L_n(\hat{\beta}_{m,n},0)=0$,  where  $L_n(\beta,\beta_0)$  is the log-likelihood function. As a consequence, for all $n$, if $\widetilde{E}_n$ happens, then the MLE for the RKHS functional logistic regression model does not exist. The result follows from the fact that $\P(E_n) = \P(\widetilde{E}_n)$ and the events $\alpha'x_i(T_n)=0$ have probability zero  since we  are assuming that the process does not have degenerate marginals.
\end{proof}

\begin{remark}\label{remark:candes}
Theorem 2.1. in \cite{candes2020phase} is a remarkable result. It applies to logistic finite-dimensional regression models with a number $p$ of covariables, which is assumed to grow to infinity with the sample size $n$, in such a way that $p/n\to\kappa$. Of course, the sample is given by data $(x_i,y_i)$, $i=1,\ldots,n$. Essentially the result establishes that there is a critical value such that, if $\kappa$ is smaller than such critical value, one has $\lim_{n,p\to\infty} \mathbb{P}(\text{MLE exists})=1$; otherwise we have $\lim_{n,p\to\infty} \mathbb{P}(\text{MLE exists})=0$. Such critical value is given in terms of a function $h$ (which is mentioned in the proof of the previous result) whose definition is as follows. Let us use the notation $(\widetilde{Y},V)\sim F_{\beta_0,\gamma_0}$ whenever $(\widetilde{Y},V)\overset{d}{=}(\widetilde{Y},\widetilde{Y}X)$, for $\widetilde{Y}=2Y-1$ (note that, in the notation of \cite{candes2020phase}, the model is defined for the case that the  response  variable is coded in $\{-1,1\}$), $\beta_0,\gamma_0\in\mathbb{R}$, $\gamma_0\geq 0$ and where $X\sim \mathcal{N}(0,1)$ and $\mathbb{P}(\widetilde{Y}=1|X) = (1+\exp\{-\beta_0-\gamma_0 X\})^{-1}$. Now, define $h(\beta_0,\gamma_0) = \min_{t_0,t_1\in\mathbb{R}}\mathbb{E}[(t_0 \widetilde{Y} + t_1 V - Z)_+^2]$, where $Z\sim \mathcal{N}(0,1)$ independent of $(\widetilde{Y},V)$ and $x_+ = \max\{x,0\}$. Then, Theorem 2.1. in \cite{candes2020phase} proves that the above mentioned critical value for $\kappa$ is precisely $h(\beta_0,\gamma_0)$.
\end{remark}

\section[The estimation of $\beta$ in practice: an RKHS motivated proposal]{ The estimation of $\boldsymbol{\beta}$ in practice} \label{sec:exPr}

The problem of non-existence of the MLE can be circumvented if the goal is variable selection. The main idea behind the proof of Theorem \ref{Teo:neasym} is that one can approximate the functional model with finite approximations as those in \eqref{eq:logfin} with $p$ increasing as fast as desired. Therefore, if we constrain $p$ to be less than a finite fixed value, Theorem \ref{Teo:neasym} does not apply.  

In order to sort out the non-existence problem  for a given sample (due to the SC property), it would be enough to use a finite-dimensional estimator that is always defined, even for linearly separable samples. As mentioned, an extensive study of existence and uniqueness conditions of the MLE for multiple logistic regression can be found in the paper of \cite{albert1984}. 

A simple, RKHS-motivated alternative would be as follows. In many cases one could assume that the ``true parameter'' $(\beta^*,\beta_0^*)$ belongs to a 
bounded set $ B_K(0,R)\times I$, $I$ being a compact interval in the real line and $B_K(0,R)$ the closed ball centered at zero, with radius $R$ in the RKHS associated with the covariance function $K$. This restriction of searching for an estimator in a ball within the parameter space resembles other regularization methods in regression such as ridge or lasso.

If $K$ is continuous and bounded, all functions $f$  in the RKHS space are continuous as well and, using the reproducing property $\langle f,K(\cdot, t)\rangle_K=f(t)$, we get
$$
\Vert f\Vert_\infty=\sup_t|\langle f,K(\cdot, t)\rangle_K|\leq \Vert f\Vert_K\sup_t K(t,t).
$$
If, for simplicity, we assume that $\sup_t K(t,t)=1$, we have (from the definition of the RKHS ${\mathcal H}(K)$) that all functions   $\beta\in B_K(0,R)$ can be approximated by functions of type
$$
g(\cdot)=\sum_{j=1}^p\beta_j K(t_j,\cdot),
$$
where $\beta_j$ are real numbers with $|\beta_j|\leq R$, $p\in{\mathbb N}$, $t_j\in [0,1]$. 

Now, recall that the RKHS functional logistic model corresponding to such function $g$ would be given by expression \eqref{eq:logfin} in terms of $\beta_i$ and $X(t_i)$. Then, assuming the continuity of the trajectories $X(t)$ we can ensure the existence of an approximate maximum likelihood (ML) estimator of  $(\beta^*,\beta_0^*)$ expressed in terms of $(\beta_0,\beta_1,\ldots,\beta_d,t_1,\ldots,t_p)$. 

The effective calculation of such estimator could be done by a sequential ``greedy'' method. The idea is to exchange the direct maximization of the likelihood function by the execution of an iterative algorithm, as follows:

\begin{enumerate}
	\item Let us fix a grid $T_p$ of $p$ equispaced points in $[0,1]$. For each $t$ on the grid, we fit the logistic model of Equation \eqref{eq:logfin} with $p=1$ and $\hat{m}(t) = \bar{X}(t)$. The log-likelihood achieved for this $t$ at the ML estimators $\wh\beta_0$ and $\wh\beta_1$ is stored in $\ell_1(t)$. Then, the first point $\wh t_1$ is fixed as the point at which $\ell_1(t)$ achieves its maximum value.
	\item Once $\wh t_1$ has been selected, for each $t$ in the grid we fit the model
	$$\P(Y=1 | X) = \bigg(1+\exp\Big\{\beta_0 + \beta_1 [X(\wh t_1)-\bar{X}(\hat{t}_1)] + \beta_2 [X(t)-\bar{X}(t)]\Big\} \bigg)^{-1}.$$
	As in the previous step, $\ell_2(t)$ would be the log-likelihood achieved at $\wh \beta_0$, $\wh\beta_1$ and $\wh\beta_2$, and $\hat{t}_2$ is the point at which the maximum of $\ell_2(t)$ is attained.
	\item We proceed in the same way until a suitable number of points $p$ has been selected.

\end{enumerate}
In practical problems, it is also important to determine how many points $p$ one should retain. The common approach is to fix this value $\wh p$ by cross-validation, whenever it is possible. Another reasonable approach is to 
increase the initial value $p$ by repeating the whole procedure with another grid $T_{p+1}$ of $p+1$ equispaced points until the increase achieved in the likelihood function is smaller than a given threshold, in a similar  way  as in \cite{berrendero2019}.

\section*{Acknowledgements}
This work has been partially supported by Spanish Grant PID2019-109387GB-I00.

\bibliography{Refs}{}

\end{document}